\documentclass[11pt]{article}
\usepackage{fullpage}
\usepackage[T1]{fontenc}
\usepackage[utf8]{inputenc}
\usepackage{hyperref}
\usepackage{url}
\usepackage{booktabs}
\usepackage{amsfonts}
\usepackage{nicefrac}
\usepackage{microtype}
\usepackage{authblk}
\usepackage{graphicx}
\usepackage{amsmath}
\usepackage{amsthm}
\usepackage{float}
\usepackage{caption}

\numberwithin{equation}{section}

\theoremstyle{plain}
\newtheorem{theorem}{Theorem}[section]
\newtheorem{lemma}[theorem]{Lemma}

\newtheorem{conjecture}[theorem]{Conjecture}
\theoremstyle{definition}

\newtheorem{definition}[theorem]{Definition}
\newtheorem{example}[theorem]{Example}

\newcommand{\cmp}{\mathbb{C}}
\newcommand{\real}{\mathbb{R}}

\newcommand{\re}{\text{Re}}

\newcommand{\ovr}{\overline}

\title{Exploiting algebraic structure in global optimization and the Belgian chocolate problem}

\author[1]{Zachary Charles}
\author[2]{Nigel Boston}
\affil[1]{Department of Mathematics, University of Wisconsin-Madison}
\affil[2]{Department of Mathematics and Department of Electrical and Computer Engineering, University of Wisconsin-Madison}

\begin{document}

\maketitle

\begin{abstract}
The Belgian chocolate problem involves maximizing a parameter $\delta$ over a non-convex region of polynomials. In this paper we detail a global optimization method for this problem that outperforms previous such methods by exploiting underlying algebraic structure. Previous work has focused on iterative methods that, due to the complicated non-convex feasible region, may require many iterations or result in non-optimal $\delta$. By contrast, our method locates the largest known value of $\delta$ in a non-iterative manner. We do this by using the algebraic structure to go directly to large limiting values, reducing the problem to a simpler combinatorial optimization problem. While these limiting values are not necessarily feasible, we give an explicit algorithm for arbitrarily approximating them by feasible $\delta$. 
Using this approach, we find the largest known value of $\delta$ to date, $\delta = 0.9808348$. 
We also demonstrate that in low degree settings, our method recovers previously known upper bounds on $\delta$ and that prior methods converge towards the $\delta$ we find.
\end{abstract}

\section{Introduction}\label{sec:intro}

	Global optimization problems of practical interest can often be cast as optimization programs over non-convex feasible regions. Unfortunately, iterative optimization over such regions may require large numbers of iterations and result in non-global maxima. Finding all or even many critical points of such programs is generally an arduous, computationally expensive task. 
	In this paper we show that by exploiting the underlying algebraic structure, we can directly find the largest known values of the Belgian chocolate problem, a famous open problem bridging optimization and control theory. Moreover, this algebraic method does not require any iterative approach. Instead of relying on eventual convergence, our method algebraically identifies points that provide the largest value of the Belgian chocolate problem so far.

	While this approach may seem foreign to the reader, we will show that our algebraic optimization method outperforms prior global optimization methods for solving the Belgian chocolate problem. We will contrast our method with the optimization method of Chang and Sahinidis \cite{chang2007global} in particular. Their method used iterative branch-and-reduce techniques \cite{ryoo1996branch} to find what was the largest known value of $\delta$ until our new approach. Due to the complicated feasible region, their method may take huge numbers of iterations or converge to suboptimal points. Our method eliminates the need for these expensive iterative computations by locating and jumping directly to the larger values of $\delta$. This approach has two primary benefits over \cite{chang2007global}. First, it allows us to more efficiently find $\delta$ as we can bypass the expensive iterative computations. This also allows us to extend our approach to cases that were not computationally tractable for \cite{chang2007global}. Second, our approach allows us to produce larger values of $\delta$ by finding a finite set of structured limit points. In low-degree cases, this set provably contains the supremum of the problem, while in higher degree cases, the set contains larger values of $\delta$ than found in \cite{chang2007global}.

	The Belgian chocolate problem is a famous open problem in control theory proposed by Blondel in 1994. In the language of control theory, Blondel wanted to determine the largest value of a process parameter for which stabilization of an unstable plant could be achieved by a stable minimum-phase controller \cite{blondel1994simultaneous}. Blondel designed the plant to be a low-degree system that was resistant to known stabilization methods, in the hope that a solution would lead to development of new stabilization techniques. Specifically, Blondel wanted to determine the largest value of $\delta > 0$ for which the transfer function $P(s) = (s^2-1)/(s^2-2\delta s+1)$ can be stabilized by a proper, bistable controller.

	For readers unfamiliar with control theory, this problem can be stated in simple algebraic terms. To do so, we will require the notion of a {\it stable} polynomial. A polynomial is stable if all its roots have negative real part. The Belgian chocolate problem is then as follows.\\
	\\
	\noindent{\bf Belgian chocolate problem:} Determine for which $\delta > 0$ there exist real, stable polynomials $x(s), y(s), z(s)$ with $\deg (x) \geq \deg (y)$ satisfying 
	\begin{equation}\label{bcp}
	z(s) = (s^2-2\delta s+1)x(s)+(s^2-1)y(s).\end{equation}
	We call such $\delta$ {\it admissible}. In general, stability of $x,y,z$ becomes harder to achieve the larger $\delta$ is. Therefore, we are primarily interested in the supremum of all admissible $\delta$. If we fix a maximum degree $n$ for $x$ and $y$, then this gives us the following global optimization problem for each $n$.\\
	\\
	\noindent{\bf Belgian chocolate problem} (optimization version):
	\begin{equation}\label{bcp_opt}
	\begin{aligned}
	& \underset{\delta, x(s), y(s)}{\text{maximize}}
	& & \delta \\
	& \text{subject to} & & x, y, z\text{ are stable},\\
	& & & z(s) = (s^2-2\delta s+1)x(s)+(s^2-1)y(s),\\
	& & & \deg(y) \leq \deg(x) \leq n.
	\end{aligned}
	\end{equation}
	Note that we can view a degree $n$ polynomial with real coefficients as a $(n+1)$-dimensional real vector of its coefficients. Under this viewpoint, the space of polynomials $x, y, z$ that are stable and satisfy (\ref{bcp}) is an extremely complicated non-convex space. As a result, it is difficult to employ global optimization methods directly to this problem. The formulation above does suggest an undercurrent of algebra in this problem. This will be exploited to transform the problem into a combinatorial optimization problem by finding points that are essentially local optima.

	Previous work has employed various optimization methods to find even larger admissible $\delta$. 
	Patel et al.~\cite{patel2002some} were the first to show that $\delta = 0.9$ is admissible by $x, y$ of degree at most 11, answering a long-standing question of Blondel.
	They further showed that $\delta = 0.93720712277$ is admissible. In 2005, Burke et al. \cite{burke2005analysis} showed that $\delta = 0.9$ is admissible with $x,y$ of degree at most 3. They also improved the record to $\delta = 0.94375$ using gradient sampling techniques. In 2007, Chang and Sahinidis used branch-and-reduce techniques to find admissible $\delta$ as large as $0.973974$ \cite{chang2007global}. In 2012, Boston used algebraic techniques to give examples of admissible $\delta$ up to 0.97646152 \cite{boston2012belgian}. Boston found polynomials that are almost stable and satisfy (\ref{bcp}). Boston then used ad hoc methods to perturb these to find stable $x,y,z$ satisfying (\ref{bcp}). While effective, no systematic method for perturbing these polynomials to find stable ones was given.

	In this paper, we extend the approach used by Boston in 2012 \cite{boston2012belgian} to achieve the largest known value of $\delta$ so far. We will refer to this method as the method of {\it algebraic specification}. We show that these almost stable polynomials serve as limiting values of the optimization program. Empirically, these almost stable polynomials achieve the supremum over all feasible $\delta$. Furthermore, we give a theoretically rigorous method for perturbing the almost stable polynomials produced by algebraic specification to obtain stable polynomials. Our approach shows that all $\delta \leq 0.9808348$ are admissible. This gives the largest known admissible value of $\delta$ to date. We further show that previous global optimization methods are tending towards the limiting values of $\delta$ found via our optimization method.

	We do not assume any familiarity on the reader's part with the algebra and control theory and will introduce all relevant notions. While we focus on the Belgian chocolate problem throughout the paper, we emphasize that the general theme of this paper concerns the underlying optimization program. We aim to illustrate that by considering the algebraic structure contained within an optimization problem, we can develop better global optimization methods.

\section{Motivation for our approach}\label{motivation}\label{sec:motivation}

	In order to explain our approach, we will discuss previous approaches to the Belgian chocolate problem in more detail. Such approaches typically perform iterative non-convex optimization in the space of stable controllers in order to maximize $\delta$. In \cite{chang2007global}, Chang and Sahinidis formulated, for each $n$, a non-convex optimization program that sought to maximize $\delta$ subject to the polynomials $x, y, (s^2-2\delta s +1)x+(s^2-1)y$ being stable and such that $n \geq \deg(x) \geq \deg(y)$. For notational convenience, we will always define $z = (s^2-2\delta s + 1)x+(s^2-1)y$. Chang and Sahinidis used branch-and-reduce techniques to attack this problem for $n$ up to 10.

	Examining the roots of the $x,y,z$ they found for $\deg(x) = 6,8,10$, a pattern emerges. Almost all the roots of these polynomials are close to the imaginary axis and are close to a few other roots. In fact, most of these roots have real part in the interval $(-0.01,0)$. In other words, the $x,y,z$ are approximated by polynomials with many repeated roots on the imaginary axis. It is also worth noting that the only roots of $x$ that were omitted are very close to $-\delta \pm \sqrt{\delta^2-1}$. This suggests that $x$ should have a factor close to $(s^2+2\delta s+1)$.

	This suggests the following approach. Instead of using non-convex optimization to iteratively push $x,y,z$ towards polynomials possessing repeated roots on the imaginary axis, we will algebraically construct polynomials with this property. This will allow us to immediately find large limit points of the optimization problem in (\ref{bcp_opt}). While the $x,y,z$ we construct are not stable, they are close to being stable. We will show later that we can perturb $x,y,z$ and thereby push their roots just to the left of the imaginary axis, causing them to be stable. This occurs at the expense of decreasing $\delta$ by an arbitrarily small amount.

	Our method only requires examining finitely many such limit points. Moreover, for reasonable degrees of $x$ and $y$, these limit points can be found relatively efficiently. By simply checking each of these limit points, we reduce to a combinatorial optimization problem. This combinatorial optimization problem provably achieves the supremal values of $\delta$ for $\deg(x) \leq 4$. For higher degree $x$, our method finds larger values of $\delta$ than any previous optimization method thus far. In the sections below we will further explain and motivate our approach, and show how this leads to the largest admissible $\delta$ found up to this point.

\section{Main results}\label{sec:math_back}

	\subsection{Preliminaries}

	Given $t \in \cmp$, we let $\re(t)$ denote its real part. We will let $\real[s]$ denote the set of polynomials in $s$ with real coefficients. For $p(s) \in \real[s]$, we call $p(s)$ {\it stable} if every root $t$ of $p$ satisfies $\re(t) < 0$. We let $H$ denote the set of all stable polynomials in $\real[s]$. We call $p(s)$ {\it quasi-stable} if every root $t$ of $p$ satisfies $\re(t) \leq 0$. We let $\overline{H}$ denote the set of quasi-stable polynomials of $\real[s]$. We let $H^m, \overline{H^m}$ denote the sets of stable and quasi-stable polynomials respectively of degree at most $m$.

	\begin{definition}We call $\delta$ {\it admissible} if there exist $x,y \in H$ such that $\deg(x) \geq \deg(y)$ and
	\begin{equation}
	(s^2-2\delta s+1)x(s) + (s^2-1)y(s) \in H.\end{equation}\end{definition}

	\begin{definition}We call $\delta$ {\it quasi-admissible} if there exist $x,y \in \ovr{H}$ such that $\deg(x) \geq \deg(y)$ and
	\begin{equation}
	(s^2-2\delta s+1)x(s) + (s^2-1)y(s) \in \ovr{H}.\end{equation}\end{definition}

	Note that since quasi-stability is weaker than stability, quasi-admissibility is weaker than admissibility. Our main theorem (Theorem \ref{main_thm} below) will show that if $\delta$ is quasi-admissible, then all smaller $\delta$ are admissible. Note that this implies that the Belgian chocolate problem is equivalent to finding the supremum of all admissible $\delta$. We will then find quasi-admissible $\delta$ in order to establish which $\delta$ are admissible. This is the core of our approach. These quasi-admissible $\delta$ are easily identified and are limit points of admissible $\delta$.

	In practice, one verifies stability by using the Routh-Hurwitz criteria. Suppose we have a polynomial $p(s) = a_0s^n + a_1s^{n-1} + \ldots + a_{n-1}s + a_n \in \real[s]$ such that $a_0 > 0$. Then we define the $n\times n$ {\it Hurwitz matrix} $A(p)$ as

	$$A(p) = \begin{pmatrix}
	a_1 & a_3 & a_5 & \ldots & \ldots & 0 & 0\\
	a_0 & a_2 & a_6 & \ldots & \ldots & 0 & 0\\
	0 & a_1 & a_3 & \ldots & \ldots & 0 & 0\\
	0 & a_0 & a_2 & \ldots & \ldots & 0 & 0\\
	\vdots & \vdots & \vdots & \ddots & \ddots & \vdots & \vdots\\
	0 & 0 & 0 & \ldots & \ldots & a_{n-2} & a_n\end{pmatrix}.$$

	Adolf Hurwitz showed that a real polynomial $p$ with positive leading coefficient is stable if and only if all leading principal minors of $A(p)$ are positive.
	While it may seem natural to conjecture that $p$ is quasi-stable if and only if all leading principal minors are nonnegative, this only works in one direction.

	\begin{lemma}Suppose $p$ is a real polynomial with positive leading coefficient. If $p$ is quasi-stable then all the leading principal minors of $A(p)$ are nonnegative.\end{lemma}
	\begin{proof}If $p(s)$ is quasi-stable, then for all $\epsilon > 0$, $p(s+\epsilon)$ is stable. Therefore, for all $\epsilon > 0$, the leading minors of $A(p(s+\epsilon))$ are all positive. Note that
	$$\lim_{\epsilon \to 0} A(p(s+\epsilon)) = A(p).$$

	Since the minors of a matrix are expressible as polynomial functions of the entries of the matrix, the leading principal minors of $A$ are limits of positive real numbers. They are therefore nonnegative.\end{proof}

	To see that the converse doesn't hold, consider $p(s) = s^4 + 198s^2 + 101^2$. Its Hurwitz matrix has nonnegative leading principal minors, but $p$ is not quasi-stable. This example, as well as a more complete characterization of quasi-stability given below, can be found in \cite{asner1970total}. In particular, it is shown in \cite{asner1970total} that a real polynomial $p$ with positive leading coefficient is quasi-stable if and only if for all $\epsilon > 0$, $A(p(s+\epsilon))$ has positive leading principal minors.

	\subsection{Quasi-admissible and admissible $\delta$}\label{delta_theory}

	We first present the following theorem concerning which $\delta$ are admissible. We will defer the proof until later as it is a simple corollary to a stronger theorem about approximating polynomials in $\overline{H}$ by polynomials in $H$.

	\begin{theorem}\label{delta_prop}If $\delta$ is admissible then all $\hat{\delta} < \delta$ are also admissible.\end{theorem}

	For $\delta = 1$, note that the Belgian chocolate problem reduces to whether there are $x,y \in H$ with $\deg(x) \geq \deg(y)$ such that $(s-1)^2x + (s^2-1)y \in H$. This cannot occur for non-zero $x,y$ since $(s-1)^2x + (s^2-1)y$ has a root at $s = 1$. Theorem \ref{delta_prop} then implies that any $\delta \geq 1$ is not admissible. In 2012, Bergweiler and Eremenko showed that any admissible $\delta$ must satisfy $\delta < 0.999579$ \cite{bergweiler2013gol}.

	On the other hand, if we fix $x,y$ then there is no single largest admissible $\delta$ associated to $x,y$. Standard results from control theory show that if $\delta$ is admissible by $x, y$ then for $\epsilon$ small enough, $\delta+\epsilon$ is admissible by the same polynomials.

	Therefore, supremum $\delta^*$ over all admissible $\delta$ will not be associated to stable $x,y$. From an optimization point of view, the associated optimization program in (\ref{bcp_opt}) has an open feasible region. In particular, the set of admissible $\delta$ for (\ref{bcp_opt}) is of the form $(0,\delta_n^*)$ for some $\delta_n^*$ that is not admissible by $x,y$ of degree at most $n$. However, as we will later demonstrate, quasi-admissible $\delta$ lie on the boundary of this feasible region. Moreover, quasi-admissible $\delta$ naturally serve as analogues of local maxima. We will therefore find quasi-admissible $\delta$ and use these to find admissible $\delta$. In Section \ref{sec:approx} we will prove the following theorem relating admissible and quasi-admissible $\delta$. The following is the main theorem of our work and demonstrates the utility of searching for quasi-admissible $\delta$.

	\begin{theorem}\label{main_thm}If $\delta$ is quasi-admissible, then all $\hat{\delta} < \delta$ are admissible. Moreover, if $\delta$ is quasi-admissible by quasi-stable $x,y$ of degree at most $n$, then any $\hat{\delta} < \delta$ is admissible by stable $\hat{x}, \hat{y}$ of degree at most $n$. \end{theorem}

	This theorem shows that to find admissible $\delta$, we need only to find quasi-admissible $\delta$. In fact our theorem will show that if $\delta$ is quasi-admissible via $x,y$ of degree at most $n$, then all $\hat{\delta} < \delta$ are admissible via $x,y$ of degree at most $n$ as well. In short, quasi-admissible $\delta$ serve as upper limit points of admissible $\delta$. Also note that since admissible implies quasi-admissible, Theorem \ref{main_thm} implies Theorem \ref{delta_prop}.

	The proof of Theorem \ref{main_thm} will be deferred until Section \ref{sec:approx}. In fact, we will do more than just prove the theorem. We will given an explicit algoritm for approximating quasi-stable $\hat{\delta}$ by stable $\delta$ within any desired tolerance. We will also be able to use the techniques in Section \ref{sec:approx} to prove the following theorem showing that admissible $\delta$ are always smaller than some quasi-admissible $\delta$.

	\begin{theorem}\label{rev_thm}If $\delta$ is admissible by $x,y$ of degree at most $n$ then there is some $\hat{\delta} > \delta$ that is quasi-admissible by $\hat{x}, \hat{y}$ of degree at most $n$. Moreover, this $\hat{\delta}$ is not admissible by these polynomials.\end{theorem}

	In other words, for any admissible $\delta$, there is a larger $\hat{\delta}$ that is quasi-admissible but not necessarily admissible. Therefore, we can restrict to looking at polynomials $x,y,z$ with at least one root on the imaginary axis.

\section{Low degree examples}\label{low_degree_ex}

	In this section we demonstrate that in low-degree settings, the supremum of all admissible $\delta$ in (\ref{bcp_opt}) is actually a quasi-admissible $\delta$. By looking at quasi-stable polynomials that are not stable, we can greatly reduce our search space and directly find the supremum of the optimization program in (\ref{bcp_opt}). 
	For small degrees of $x, y$, we will algebraically design quasi-stable polynomials that achieve previously known bounds on the Belgian chocolate problem in these degrees.

	Burke et al. \cite{burke2005analysis} showed that for $x\in H^3, y \in H^0$, any admissible $\delta$ must satisfy $\delta < \sqrt{2+\sqrt{2}}/2$ and for $x \in H^4, y \in H^0$, $\delta$ must satisfy $\delta < \sqrt{10+2\sqrt{5}}/4$. He et al. \cite{guannan2007stabilization} later found $x \in H^4, y \in H^0$ admitting $\delta$ close to this bound.

	In fact, these upper bounds on admissible $\delta$ are actually quasi-admissible $\delta$ that can be obtained in a straightforward manner. For example, suppose we restrict to $x$ of degree 3, $y$ of degree 0. Then for some $A, B, C, k \in \real$, we have
	\begin{gather*}
	x(s) = s^3+ As^2 + Bs + C\\
	y(s) = k\end{gather*}

	Instead of trying to find admissible $\delta$ using this $x$ and $y$, we will try to find quasi-admissible $\delta$. That is, we want $\delta$ such that
	$$z(s) = (s^2-2\delta s + 1)x(s) + (s^2-1)y(s) \in \overline{H}.$$

	In other words, this $z(s)$ can be quasi-stable instead of just stable. Note that $z(s)$ must be of degree 5. We will specify a form for $z(s)$ that ensures it is quasi-stable. Consider the case $z(s) = s^5$. This is clearly quasi-stable as its only roots are at $s = 0$. To ensure that $z(s) = s^5$ and equation (\ref{bcp}) holds, we require
	\begin{gather*}
	(s^2-2\delta s + 1)(s^3+ As^2 + Bs + C)+(s^2-1)k = s^5\end{gather*}

	Equating coefficients gives us the following 5 equations in 5 unknowns.
	\begin{gather*}
	A-2\delta=0\\
	-2A\delta + B + 1=0\\
	A - 2B\delta + C + k=0\\
	B-2C\delta=0\\
	C-k=0\end{gather*}

	In fact, ensuring that we have as many equations as unknowns was part of the motivation for letting $z(s) = s^5$. Solving for $A,B,C,k,\delta$, we find
	\begin{gather*}
	8\delta^4-8\delta^2+1=0\\
	A = 2\delta\\
	B = 4\delta^2-1\\
	C = 4\delta^3-2\delta\\
	k = 4\delta^3-2\delta\end{gather*}

	Taking the largest real root of $8\delta^4-8\delta^2+1$ gives $\delta = \sqrt{2+\sqrt{2}}/2$. Taking $A,B,C,k$ as above yields polynomials $x, y, z$ with real coefficients. One can verify that $x$ is stable (via the Routh-Hurwitz test, for example), while $y$ is degree 0 and therefore stable. Note that since $z(s) = s^5, z$ is only quasi-stable. Therefore, there is $x \in H^3, y \in H^0$ for which $\sqrt{2+\sqrt{2}}/2$ is quasi-admissible. This immediately gives the limiting value for $x \in H^3, y \in H^0$ discovered by Burke et al \cite{burke2005analysis}. Combining this with Theorem \ref{main_thm}, we have shown the following theorem.
	\begin{theorem}For $\deg(x) \leq 3$, $\delta = \frac{\sqrt{2+\sqrt{2}}}{2}$ is quasi-admissible and all $\delta < \frac{\sqrt{2+\sqrt{2}}}{2}$ are admissible.\end{theorem}
	Next, suppose that $x$ has degree 4 and $y$ has degree 0. For $A, k, \delta \in \real$, define
	\begin{gather*}
	x(s) = (s^2+2\delta s + 1)(s^2+A)\\
	y(s) = k\end{gather*}

	Note that as long as $A \geq 0$, $x$ will be quasi-stable and $y$ will be stable for any $k$. As above, we want quasi-admissible $\delta$. We let $z(s) = s^6$, so that $z(s)$ is quasi-stable. Finding $A, \delta, k$ amounts to solving
	\begin{gather*}
	(s^2-2\delta s + 1)x(s) + (s^2-1)y(s) = z(s)\\
	\Leftrightarrow (s^2-2\delta s + 1)(s^2+2\delta s+1)(s^2+A)+(s^2-1)k = s^6\\
	\Leftrightarrow s^6 + (A - 4\delta^2 + 2)s^4 + (-4A\delta^2 + 2A + k + 1)s^2 + (A-k) = s^6
	\end{gather*}

	Note that the $(s^2+2\delta s + 1)$ term in $x$ is used to ensure that the left-hand side will have zero coefficients in its odd degree terms. Since $(s^2+2\delta s + 1)$ is stable, it does not affect stability of $x$. Equating coefficients and manipulating, we get the following equations.
	\begin{gather*}
	16\delta^4 -20\delta^2+5=0\\
	A -4\delta^2+2=0\\
	k -A=0\end{gather*}

	Taking the largest real root of $16\delta^4 -20\delta^2+5$ gives $\delta = \sqrt{10+2\sqrt{5}}/4$. For this $\delta$ one can easily see that $A = 4\delta^2 - 2 \geq 0$, so $x$ is quasi-stable, as are $y$ and $z$ by design. Once again, we were able to easily achieve the limiting value discovered by Burke et al. \cite{burke2005analysis} discussed in Section \ref{low_degree_ex} by searching for quasi-admissible $\delta$. Combining this with Theorem \ref{main_thm}, we obtain the following theorem.

	\begin{theorem}For $\deg(x) \leq 4$, $\delta  = \frac{\sqrt{10+2\sqrt{5}}}{4}$ is quasi-admissible and all $\delta < \frac{\sqrt{10+2\sqrt{5}}}{4}$ are admissible.\end{theorem}

	The examples above demonstrate how, by considering quasi-stable $x,y$ and $z$, we can find quasi-admissible $\delta$ that are limiting values of admissible $\delta$. Moreover, the quasi-stable $\delta$ above were found by solving relatively simple algebraic equations instead of having to perform optimization over the space of stable $x$ and $y$.

\section{Algebraic specification}\label{sec:alg_spec}

	The observations in Section \ref{sec:motivation} and Section \ref{sec:math_back} and the examples in Section \ref{low_degree_ex} suggest the following approach which we refer to as {\it algebraic specification}. This method will be used to find the largest known values of $\delta$ found for any given degree. We wish to construct quasi-stable $x(s), y(s), z(s)$ with repeated roots on the imaginary line satisfying (\ref{bcp}). For example, we may wish to find polynomials of the following form:
	\begin{gather*}
	x(s) = (s^2+2\delta s+1)(s^2+A_1)^4(s^2+A_2)^2(s^2+A_3)^2(s^2+A_4)\\
	y(s) = k(s^2+B_1)^3(s^2+B_2)^2\\
	z(s) = s^{14}(s^2+C_1)^2(s^2+C_2)(s^2+C_3)\end{gather*}

	We refer to such an arrangement of $x,y,z$ as an {\it algebraic configuration}. As long as $\delta > 0$, the parameters $\{A_i\}_{i=1}^4$, $\{B_i\}_{i=1}^2$, and $\{C_i\}_{i=1}^3$ are all nonnegative, and $k$ is real, $x(s), y(s), z(s)$ will be real, quasi-stable polynomials. We then wish to solve
	\begin{equation}\label{alg_eq}
	(s^2-2\delta s+1)x(s)+(s^2-1)y(s)=z(s)\end{equation}

	Recall that the $(s^2+2\delta s+1)$ factor in $x(s)$ is present to ensure that the left-hand side has only even degree terms, as the right-hand side clearly only has even degree terms. Expanding (\ref{alg_eq}) and equating coefficients, we get 11 equations in 11 unknowns. Using PHCPack \cite{verschelde1999algorithm} to solve these equations and selecting the solution with the largest $\delta$ such that the $A_i, B_i, C_i \geq 0$, we get the following solution, rounded to seven decimal places:
	\begin{gather*}
	\delta = 0.9808348\\
	A_1 = 1.1856917\\
	A_2 = 6.6228807\\
	A_3 = 0.3090555\\
	A_4 = 0.2292503\\
	B_1 = 0.5430391\\
	B_2 = 0.2458118\\
	C_1 = 4.4038385\\
	C_2 = 0.7163490\\
	C_3 = 7.4637156\\
	k = 196.1845537
	\end{gather*}

	The actual solution has $\delta = 0.980834821202\ldots$. This is the largest $\delta$ we have found to date using this method. By Theorem \ref{main_thm}, we conclude the following theorem.

	\begin{theorem}All $\delta \leq 0.9808348$ are admissible.\end{theorem}

	In general, we can form an algebraic configuration for $x(s), y(s), z(s)$ as
	\begin{equation}\label{xconf}
	x(s) = (s^2+2\delta s +1) \prod_{i=1}^{m_1} (s^2+A_i)^{j_i}.\end{equation}
	\begin{equation}\label{yconf}
	y(s) = k\prod_{i=1}^{m_2}(s^2+B_i)^{k_i}.\end{equation}
	\begin{equation}\label{zconf}
	z(s) = s^c\prod_{i=1}^{m_3}(s^2+C_i)^{\ell_i}.\end{equation}

	For fixed degrees of $x, y$, note there are only finitely many such configurations. Instead of performing optimization over the non-convex feasible region of the Belgian chocolate problem, we instead tackle the combinatorial optimization problem of maximizing $\delta$ among the possible configurations. 

	Note that $c$ in (\ref{zconf}) is whatever exponent is needed to make $\deg(z) = \deg(x)+2$. We want $x,y,z$ to satisfy (\ref{bcp}). Expanding and equating coefficients, we get equations in the undetermined variables above. As long as the number of unknown variables equals the number of equations, we can solve and look for real solutions with $\delta$ and all $A_i, B_i, C_i$ nonnegative.

	Not all quasi-stable polynomials can be formed via algebraic specification. In particular, algebraic specification forces all the roots of $y,z$ and all but two of the roots of $x$ to lie on the imaginary axis. However, more general quasi-stable $x,y,z$ could have some roots with negative real part and some with zero real part. This makes the possible search space infinite and, as discussed in Section \ref{low_degree_ex}, empirically does not result in larger $\delta$. Further evidence for this statement will be given in Section \ref{sec:opt}.

	While the method of algebraic specification has demonstrable effectiveness, it becomes computationally infeasible to solve these general equations for very large $n$. In particular, the space of possible algebraic configurations of $x,y,z$ grows almost exponentially with the degree of the polynomials. For large $n$, an exhaustive search over the space of possible configurations becomes infeasible, especially as the equations become more difficult to solve.

	We will describe an algebraic configuration via the shorthand 
	\begin{equation}\label{alg_conf_shorthand}
	[j_1,\ldots,j_{m_1}],[k_1,\ldots, k_{m_2}],[\ell_1,\ldots, \ell_{m_3}].\end{equation}
	This represents the configuration described in $(\ref{xconf}),(\ref{yconf}),(\ref{zconf})$ above. In particular, if the second term of (\ref{alg_conf_shorthand}) is empty then $y = k$, while if the third term of (\ref{alg_conf_shorthand}) is empty then $z$ is a power of $s$. For example, the following configuration is given by $[3,1],[2],[1]$:
	\begin{gather*}
	x(s) = (s^2+2\delta s + 1)(s^2+A_1)^3(s^2+A_2)\\
	y(s) = k(s^2+B_1)^2\\
	z(s) = s^{10}(s^2+C_1)\end{gather*}

	A table containing the largest quasi-admissible $\delta$ we have found and their associated algebraic configuration for given degrees of $x$ is given below. Note that for each entry of the table, given $\deg(x) = n$ and quasi-admissible $\delta$, Theorem \ref{main_thm} implies that all $\hat{\delta} < \delta$ are admissible with $x,y$ of degree at most $n$.

	\begin{figure}[H]
		\centering
		{
		\begin{tabular}{ |c|c|c| } 
		\hline
		$\deg(x)$ & Configuration & $\delta$\\
		\hline
		4 & [1],[],[] & 0.9510565 \\
		\hline
		6 & [2],[1],[] & 0.9629740 \\
		\hline
		8 & [3],[1],[1] & 0.9702883 \\
		\hline
		10 & [3,1],[2],[1] & 0.9744993\\
		\hline
		12 & [3,2],[2,1],[1] & 0.9764615 \\
		\hline
		14 & [3,2,1],[2,1],[2] & 0.9783838 \\
		\hline
		16 & [3,2,1,1],[2,2],[2] & 0.9794385\\
		\hline
		18 & [3,2,2,1],[2,2],[2,1] & 0.9802345 \\
		\hline
		20 & [4,2,2,1],[3,2],[2,2,1] & 0.9808348 \\
		\hline
		\end{tabular}
		}
		\caption{The largest known quasi-admissible $\delta$ for $x,y,z$ designed algebraically, for varying degrees of $x$.}
	\end{figure}

\section{Approximating quasi-admissible $\delta$ by admissible $\delta$}\label{sec:approx}

	In this section we will prove Theorem \ref{main_thm}. Our proof will be algorithmic in nature. We will describe an algorithm that, given $\delta$ that is quasi-admissible by quasi-stable polynomials $x, y$, will produce for any $\hat{\delta} < \delta$ stable polynomials $\hat{x}, \hat{y}$ admitting $\hat{\delta}$. Moreover, given $\deg(x) = n$, we will ensure that $\deg(\hat{x}) \leq n$.

	\begin{proof}[of Theorem \ref{main_thm}]Suppose that for a given $\delta$ there are $x,y,z \in \ovr{H}$ with $\deg(x) \geq \deg(y)$ satisfying (\ref{bcp}). Let $n = \deg(x)$. Define
	$$R(s) := \dfrac{(s^2-1)y(s)}{z(s)}.$$
	
	Note that for any $s \in \cmp$, $R(s) = 0$ iff $(s^2-1)y(s) = 0$, $R(s) = 1$ iff $(s^2-2\delta s+1)x(s) = 0$, and $R(s)$ is infinite iff $z(s) = 0$. Since $x,y,z$ are quasi-stable, we know that for $\re(s) > 0$, $R(s) = 1$ iff $s = \delta \pm i\sqrt{1-\delta^2}$ and $R(s) = 0$ iff $s = 1$. All other points where $R(s)$ is 0, 1, or infinite satisfy $\re(s) \leq 0$.
	Precomposing $R(s)$ with the fractional linear transformation $f(s) = (1+s)/(1-s)$, we get the complex function
	$$D(s) := R\bigg(\dfrac{1+s}{1-s}\bigg).$$

	Note that this fractional linear transformation maps the unit disk $\{s | |s| = 1\}$ to the imaginary axis $\{ s | \re(s) = 0\}$. Also note that $f^{-1}(1) = 0, f^{-1}(\delta \pm i\sqrt{1-\delta^2}) = \pm it$ where $t = \sqrt{1-\delta}/\sqrt{1+\delta}$. Therefore, $D(s)$ satisfies the following properties:
	\begin{enumerate}
		\item For $|s| < 1$, $D(s) = 0$ iff $s = 0$.
		\item For $|s| < 1$, $D(s) = 1$ iff $s = \pm it$.
		\item $|D(s)| < \infty$ for $|s| < 1$.
	\end{enumerate}

	Note that the last holds by the quasi-stability of $z(s)$. Since $z(s) = 0$ implies $\re(s) \leq 0$, $D(s) = \infty$ implies $|s| \geq 1$. In particular, the roots of $x, y, z$ that have 0 real part now correspond to points $|s| = 1$ such that $D(s) = 1, 0, \infty$ respectively. For any $\epsilon > 0$, let
	$$D_\epsilon(s) := D\bigg(\frac{s}{1+\epsilon}\bigg).$$
	$D_\epsilon(s)$ then satisfies
	\begin{enumerate}
		\item For $|s| \leq 1$, $D_\epsilon(s) = 0$ iff $s = 0$.
		\item For $|s| \leq 1$, $D_\epsilon(s) = 1$ iff $s = \pm i(1+\epsilon)t$.
		\item $|D(s)| < \infty$ for $|s| \leq 1$.
	\end{enumerate}
	Precomposing with the inverse fractional linear transformation $f^{-1}(s) = (s-1)/(s+1)$, we get
	$$R_\epsilon(s) := D_\epsilon\bigg(\dfrac{s-1}{s+1}\bigg).$$
	By the properties of $D_\epsilon(s)$ above, we find that $R_\epsilon(s)$ satisfies
	\begin{enumerate}
		\item For $\re(s) \geq 0$, $R_\epsilon(s) = 0$ iff $s = 1$.
		\item For $\re(s) \geq 0$, $R_\epsilon(s) = 1$ iff $s = \delta_\epsilon\pm i\sqrt{1-\delta_\epsilon^2}$ where
		$$\delta_\epsilon = \dfrac{1-(1+\epsilon)^2t^2}{1+(1+\epsilon^2)t^2}.$$
		\item For $\re(s) \geq 0$, $|R_\epsilon(s)| < \infty$.
	\end{enumerate}

	Moreover, $R_\epsilon(s) \neq 0, 1, \infty$ for any $s$ such that $\re(s) < 0$. We can rewrite $R_\epsilon(s)$ as $R_\epsilon(s) = p(s)/q(s)$.
	Note that by the first property of $R_\epsilon$, the only root of $p(s)$ in $\{s | \re(s) \geq 0\}$ is at $s = 1$. By properties of $f(s), f^{-1}(s)$, one can show that $p(-1) = 0$. This follows from the fact that $R(-1) = 0$, which implies that $\lim_{s\to \infty} D(s) = \lim_{s\to\infty}D_{\epsilon}(s) = 0$, and therefore $R_\epsilon(-1)  = 0$. Therefore, $p(s) = (s^2-1)y_\epsilon(s)$ where $y_\epsilon(s)$ has no roots in $\{s | \re(s) \geq 0\}$.
	By the second property of $R_\epsilon$, the only roots of $q-p$ in $\{s | \re(s) \geq 0\}$ are at $\pm \delta_\epsilon + i\sqrt{1-\delta_\epsilon^2}$. Therefore, $q-p = (s^2-2\delta_\epsilon s+1)x_\epsilon(s)$ where $x_\epsilon(s)$ has no roots in $\{s | \re(s) \geq 0\}$.
	Finally, by the third property of $R_\epsilon$ we find that $z_\epsilon(s) = (s^2-2\delta_\epsilon s+1)x_\epsilon(s)+(s^2-1)y_\epsilon(s)$ is stable. Moreover, basic properties of fractional linear transformations show that if $\deg (x) = n \geq \deg(y) = m$, then $x_\epsilon, y_\epsilon$ are both of degree $n$. Therefore, $x_\epsilon, y_\epsilon, z_\epsilon$ are stable polynomials satisfying (\ref{bcp}) for $\delta_\epsilon$. For any $\hat{\delta} < \delta$, we can take $\epsilon$ such that $\delta_\epsilon = \hat{\delta}$, proving the desired result.\end{proof}

	Note that if we start with $\delta$ admissible by stable $x,y,z$ of degree at most $n$, then we can do the reverse of this procedure to perturb $x,y,z$ to quasi-stable $\hat{x}, \hat{y}, \hat{z}$. By the reverse of the arguments above, $\hat{x}, \hat{y}, \hat{z}$ will be quasi-stable but at least one of these polynomials will not be stable. These polynomials will be associated to some quasi-admissible $\hat{\delta} > \delta$. This gives the proof of Theorem \ref{rev_thm}.

	The proof above describes the following algorithm for perturbing quasi-stable $x,y,z$ satisfying (\ref{bcp}) to obtain stable $\hat{x}, \hat{y},\hat{z}$ satisfying (\ref{bcp}).\\
	\\
	\noindent{\bf Input:} Real numbers $\delta, \epsilon > 0$ and real polynomials $x,y,z \in \ovr{H}$ satisfying (\ref{bcp}).\\
	\noindent{\bf Output:} $\hat{\delta}$ and real polynomials $\hat{x},\hat{y},\hat{z} \in H$ satisfying (\ref{bcp}).
	\begin{enumerate}
		\item Let $R(s) = (s^2-1)y(s)/z(s)$. For $\epsilon > 0$, compute
		$$R_\epsilon(s) = R\bigg(\dfrac{(2+\epsilon)s + \epsilon}{\epsilon s + (2+\epsilon)} \bigg).$$
		\item Reduce $R_\epsilon(s)$ to lowest terms. Suppose that in lowest terms $R_\epsilon(s) = p(s)/q(s)$.
		\item Factor $p(s)$ as $(s^2-1)\hat{y}(s)$ and factor $q(s)-p(s)$ as $(s^2-2\hat{\delta}s+1)\hat{x}(s)$. Let $\hat{z}(s) = q(s)$.
	\end{enumerate}

	\medskip

	To further illustrate the method of algebraic specification and this algorithm for perturbing to get quasi-stable polynomials, we give the following detailed example.

	\begin{example}Say we are interested in $x$ of degree 4. We may then give the following algebraic specification of $x, y, z$ discussed in Section \ref{low_degree_ex}. In the shorthand of (\ref{alg_conf_shorthand}), this is the configuration $[1],[],[]$.
	\begin{gather*}
	x(s) = (s^2+2\delta s+1)(s^2+A)\\
	y(s) = k\\
	z(s) = s^6\end{gather*}
	As in Section \ref{low_degree_ex}, we solve $(s^2-2\delta s+1)x(s) + (s^2-1)y(s) = z(s)$. This implies that $\delta, A, k$ satisfy $16\delta^4-20\delta^2+5 = 0$, $A = 4\delta^2-2$, $k = 4\delta^2 - 2$. Taking the largest root of $16\delta^4-20\delta^2+5$ gives $\delta = \sqrt{10+2\sqrt{5}}/4$, $A = k = (\sqrt{5}+1)/2$. Given numerically to six decimal places, $\delta = 0.951057$. Computing $R(s)$ using exact arithmetic, we get
	\begin{gather*}
	R(s) = \dfrac{(s^2-1)y(s)}{z(s)} = \dfrac{(s^2-1)(\sqrt{5}+1)}{2s^6}\end{gather*}
	We then use a fractional linear transformation $s \mapsto (1+s)/(1-s)$ to get:
	\begin{align*}
	D(s) &= R((1+s)/(1-s))\\
	&= \dfrac{2s(\sqrt{5}+1)(s-1)^4}{s^6+6s^5+15s^4+20s^3+15s^2+6s+1}\end{align*}

	One can verify that $D(s)$ can equal 1 on the boundary of the unit circle, so we push these away from the boundary (with $\epsilon = 0.01$) by defining
	\begin{align*}
	D_\epsilon(s) &=D\big(\frac{s}{1+0.01}\big)\\
	&=\dfrac{6.40805(0.99010s-1)^4s}{0.942045s^6+\ldots+5.94054s}\end{align*}

	While we gave an approximate decimal form above for brevity, this computation can and should be done with exact arithmetic. We let $R_\epsilon(s) = f_\epsilon((s-1)/(s+1))$. Writing $R_\epsilon(s)$ as $p(s)/q(s)$ in lowest terms, we get:
	\begin{gather*}
	p(s) = 64080.55401(0.990990s+199.00990)^4(s^2-1)\\
	q(s) = 0.62122\times 10^{14}s^6 + \ldots +0.94204\end{gather*}

	As proved above, $p(s)$ will equal $(s^2-1)\hat{y}(s)$. Dividing $p(s)$ by the $s^2-1$ factor, we get a polynomial $\hat{y}(s)$ such that its only root is at $s = -201$. Therefore $\hat{y}(s)$ is stable. The denominator, $\hat{z}(s)$ is easily verified to only have roots with negative part. Finally, the polynomial $q(s) - p(s)$ will equal $(s^2-2\hat{\delta}s+1)\hat{x}(s)$. Finding its roots, one can show that $q(s)-p(s)$ only has roots with negative real part, except for roots at $s = 0.950097 \pm 0.311954i$. These roots are of the form $\hat{\delta} \pm \sqrt{\hat{\delta}^2-1}$ for $\hat{\delta} = 0.950097$. Therefore $\hat{\delta} = 0.950097$ is admissible via the stable polynomials $\hat{x},\hat{y},\hat{z}$. While we have decreased $\delta$ slightly, we have achieved stability in the process. By decreasing $\epsilon$, we can get arbitrarily close to our original $\delta$.
	\end{example}

\section{Optimality of algebraic specification}\label{sec:opt}

	Not only does our method of algebraic specification find larger $\delta$ than have been found before, one can view previous approaches to the Belgian chocolate problem as approximating algebraic specification. In particular, previously discovered admissible $\delta$ can be seen as approximating some quasi-admissible $\delta'$ that can be found via algebraic specification.

	For example, in \cite{chang2007global}, Chang and Sahinidis found that $\delta = 0.9739744$ is admissible by
	\begin{align*}
	x(s) &=s^{10} + 1.97351109136261s^9\\
	&+5.49402092964662s^8 + 8.78344232801755s^7\\
	&+ 11.67256448604672s^6 + 13.95449016040116s^5\\
	&+11.89912895529042s^4 + 9.19112429409894s^3\\
	&+5.75248874640322s^2+2.03055901420484s\\
	&+1.03326203778346,\\
	y(s)&=0.00066128189295s^5+3.611364710425s^4\\
	&+0.03394722108511s^3+3.86358782861648s^2\\
	&+0.0178174691792s+1.03326203778319.\\
	\end{align*}

	The roots of $x,y,z$ were discussed in Section \ref{motivation}. As previously noted, $x,y,z$ are close to polynomials with repeated roots on the imaginary axis. Examining the roots of $x,y,z$, one can see that $x,y,z$ are tending towards quasi-stable polynomials $x', y', z'$ that have the same root structure as the algebraic configuration $[3,1],[2],[1]$. In other words, we will consider the following quasi-stable polynomials:

	\begin{gather*}
	x'(s) = (s^2+2\delta' s + 1)(s^2+A_1)^3(s^2+A_2)\\
	y'(s) = k(s^2+B)^2\\
	z'(s) = s^{10}(s^2+C)
	\end{gather*}

	Solving for the free parameters and finding the largest real $\delta'$ such that $A_1, A_2, B, C \geq 0$, we obtain the following values, given to seven decimal places.
	\begin{gather*}
	\delta' = 0.9744993\\
	A_1 =  1.3010813\\
	A_2 =  0.4475424\\
	B =    0.5345301\\
	C =    2.5521908\\
	k =    3.4498736.\end{gather*}

	One can easily verify that taking these values of the parameters, the roots of $x, y, z$ are close to the roots of $x',y',z'$. These algebraically designed $x', y', z'$ possess the root structure that $x,y,z$ are tending towards. Moreover, the $x', y', z'$ show that $\delta'$ is quasi-stable and their associated $\delta'$ gives an upper bound for the $\delta$ found by Chang and Sahinidis. This demonstrates that the stable polynomials found by Chang and Sahinidis are tending towards the quasi-stable ones listed above. Moreover, by Theorem \ref{main_thm} all $\delta < 0.9744993$ are admissible.

	In fact, many examples of admissible $\delta$ given in previous work are approximating quasi-admissible $\delta$ found via algebraic specification. This includes the previously mentioned examples in \cite{burke2005analysis} and all admissible values of $\delta$ given by Chang and Sahinidis in \cite{chang2007global}. We further conjecture that for all admissible $\delta$, there is a quasi-admissible $\delta' > \delta$ that can be achieved by algebraically specified $x,y,z$.

	More formally, if we fix $x, y$ to be of degree at most $n$, let $\delta_n^*$ denote the supremum of the optimization problem in (\ref{bcp_opt}). Note that as discussed in Section \ref{delta_theory}, $\delta_n^*$ is not admissible by $x,y$ of degree at most $n$. The empirical evidence given in this section and in Sections \ref{sec:motivation} and \ref{low_degree_ex} suggests that this $\delta_n^*$ is quasi-admissible and can be obtained through algebraic specification. This leads to the following conjecture.

	\begin{conjecture}For all $n$, $\delta_n^*$ is quasi-admissible by some $x,y,z$ that are formed via algebraic specification.\end{conjecture}

\section{Conclusion}

	The Belgian chocolate problem has remained resilient to direct global optimization techniques for over a decade. Most prior work attempts to maximize $\delta$ subject to the stability constraints by applying iterative methods to complicated non-convex regions. By contrast, we find the largest known value of $\delta$ in a more direct fashion. We do this by reducing our problem to combinatorial optimization over a finite set of algebraically constructed limit points. Our key algebraic insight is that quasi-admissible $\delta$ are limiting values of the admissible $\delta$. In fact, previous methods actually find admissible $\delta$ that approach quasi-admissible $\delta$. We give the method of algebraic specification to design quasi-stable polynomials and directly find these quasi-admissible $\delta$ by solving a system of equations. We then show that we can perturb these quasi-stable polynomials to obtain stable polynomials with admissible $\delta$ that are arbitrarily close to the quasi-admissible $\delta$. We show that this method recovers the largest admissible $\delta$ known to date and gives a much better understanding of the underlying landscape of admissible and quasi-admissible $\delta$. We conjecture that for all $n$, the supremum of all $\delta$ admissible by $x,y$ of degree at most $n$ is a quasi-admissible $\delta$ that can be found through our method of algebraic specification.

\section*{Acknowledgments}

	The authors would like to thank Bob Barmish for his valuable feedback, discussions, and advice. The first author was partially supported by the National Science Foundation grant DMS-1502553. The second author was partially supported by the Simons Foundation grant MSN179747.

\bibliographystyle{spmpsci}

\bibliography{refs}

\end{document}